\documentclass[12pt]{scrartcl}

\usepackage[]{amscd,amsmath, amssymb,amsfonts,amsthm,mathtools,braket,color,enumerate}

\usepackage{cite}
\usepackage{hhline}
\usepackage{url}
\usepackage{here}
\usepackage{tikz-cd}
\usepackage{tikz}
\usepackage{hyperref}
\hypersetup{
  colorlinks   = true, 
  urlcolor     = blue, 
  linkcolor    = blue, 
  citecolor   = red 
}
\usetikzlibrary{decorations}
\usetikzlibrary{arrows}
\tikzstyle{v} = [circle, draw, inner sep=2pt, minimum size=3pt, fill=black]

\theoremstyle{plain}
\newtheorem{theorem}{Theorem}[section]
\newtheorem{lemma}[theorem]{Lemma}
\newtheorem{proposition}[theorem]{Proposition}

\theoremstyle{definition}
\newtheorem{definition}[theorem]{Definition}

\newtheorem{example}[theorem]{Example}
\newtheorem{problem}[theorem]{Problem}

\begin{document}

\title{Upper Embeddability of Graphs and Products of Transpositions Associated with Edges}
\author{
Shuhei Tsujie
\thanks{Department of Mathematics, Hokkaido University of Education, Asahikawa, Hokkaido 070-8621, Japan. 
E-mail:tsujie.shuhei@a.hokkyodai.ac.jp}
\and
Ryo Uchiumi
\thanks{
Department of Mathematics, Graduate School of Science, Osaka University, Toyonaka, Osaka 560-0043, Japan. 
E-mail:uchiumi.ryou.1xu@ecs.osaka-u.ac.jp
}
}

\date{}

\maketitle

\begin{abstract}
Given a graph, we associate each edge with the transposition which exchanges the endvertices. 
Fixing a linear order on the edge set, we obtain a permutation of the vertices. 
D\'enes proved that the permutation is a full cyclic permutation for any linear order if and only if the graph is a tree. 

In this article, we characterize graphs having a linear order such that the associated permutation is a full cyclic permutation in terms of graph embeddings. 
Moreover, we give a counterexample for Eden's question about an edge ordering whose associated permutation is the identity. 
\end{abstract}



\tableofcontents

\section{Introduction}

In this article, a graph $ G $ stands for a \textbf{connected multigraph} $ G = (V_{G}, E_{G}, r_{G}) $, where 
\begin{itemize}
\item $ V_{G} $ is a finite set of vertices, 
\item $ E_{G} $ is a finite set of edges, 
\item $ r_{G} $ is a map from $ E_{G} $ to $ \binom{V_{G}}{2} $, the collection of subsets of $ V_{G} $ consisting of $ 2 $ elements. 
\end{itemize}
Note that $ r_{G}(e) $ represents the endvertices of the edge $ e $. 
Also note that loops are not allowed. 

Let $ n $ be a positive integer and suppose that $ V_{G} = [n] \coloneqq \{1, 2, \dots, n\} $. 
When an edge $ e \in E_{G} $ satisfies $ r_{G}(e) = \{u,v\} $, we associate the transposition $ \tau_{e} \coloneqq (u \ v) \in S_{n} $ with the edge $ e $, where $ S_{n} $ denotes the symmetric group of degree $ n $. 

An \textbf{edge ordering} of a graph $ G $ is a linear order $ \leq_{\omega} $ on $ E_{G} $, denoted as a sequence $ \omega = (e_{1}, \dots, e_{m}) $ in which $ e_{i} <_{\omega} e_{j} $ if and only if $ i < j $. 
Given an edge ordering $ \omega = (e_{1}, \dots, e_{m}) $, we associate the product $ \pi_{\omega} \coloneqq \tau_{e_{m}} \cdots \tau_{e_{1}} \in S_{n} $. 

\begin{definition}
A permutation $ \sigma \in S_{n} $ is called a \textbf{full cyclic permutation} if $ \sigma $ is a cyclic permutation of length $ n $. 
An edge ordering $ \omega = (e_{1}, \dots, e_{m}) $ of a graph $ G $ is a \textbf{full cyclic permutation ordering} if the corresponding permutation $ \pi_{\omega} = \tau_{e_{m}} \cdots \tau_{e_{1}} $ is a full cyclic permutation. 
\end{definition}

D{\'e}nes proved the following theorem to state a connection between labeled trees and factorization of a full cyclic permutation into transpositions. 
\begin{theorem}[D{\'e}nes \cite{denes1959representation-potmiothaos}. See also {\cite[Section 2]{moszkowski1989solution-ejoc}} and {\cite[Lemma 2.1]{pawlowski2022chromatic-ac}}]\label{Denes}
Given a graph $ G $, the following are equivalent. 
\begin{enumerate}[(i)]
\item Any edge ordering of $ G $ is a full cyclic permutation ordering. 
\item $ G $ is a tree. 
\end{enumerate}
\end{theorem}

\begin{figure}[t]
\centering
\begin{minipage}[b]{0.45\linewidth}
\centering
\begin{tikzpicture}
\draw (-1.6,  0.5) node[v](1){} node[above]{$ 1 $};
\draw (-1.6, -0.5) node[v](2){} node[below]{$ 2 $};
\draw (     0,    0) node[v](3){} node[above]{$ 3 $};
\draw ( 1.6,  0.5) node[v](4){} node[above]{$ 4 $};
\draw ( 1.6, -0.5) node[v](5){} node[below]{$ 5 $};
\draw (1)--(2) node[midway, xshift=-8, yshift=0]{$ e_{1} $};
\draw (2)--(3) node[midway, xshift=2, yshift=-6]{$ e_{2} $};
\draw (3)--(4) node[midway, xshift=-2, yshift=6]{$ e_{3} $};
\draw (4)--(5) node[midway, xshift=8, yshift=0]{$ e_{4} $};
\draw (1)--(3) node[midway, xshift=2, yshift=6]{$ e_{5} $};
\draw (3)--(5) node[midway, xshift=-2, yshift=-6]{$ e_{6} $};
\end{tikzpicture}
\caption{A butterfly graph.}
\label{fig: butterfly graph}
\end{minipage}
\begin{minipage}[b]{0.45\linewidth}
\centering
\begin{tikzpicture}
\draw (-1.6,  0.5) node[v](1){};
\draw (-1.6, -0.5) node[v](2){} node[below]{\rule{0mm}{3mm}};
\draw (   0,    0) node[v](3){};
\draw (   1.6,    0) node[v](4){};
\draw ( 3.2,  0.5) node[v](5){};
\draw ( 3.2, -0.5) node[v](6){};
\draw (3)--(1)--(2)--(3)--(4)--(5)--(6)--(4);
\end{tikzpicture}
\caption{A dumbbell graph.}
\label{fig: dumbbell graph}
\end{minipage}
\end{figure}

Note that Theorem \ref{Denes} plays an important role in the studies of the chromatic symmetric functions and the chromatic operator for trees \cite{pawlowski2022chromatic-ac,foley2022transplanting}. 
Also, note that recently the second author \cite{uchiumi2023signed-ejogtaa} studied an analogue of Theorem \ref{Denes} for signed graphs and the hyperoctahedral group. 

It is a natural question to ask what graphs admit a full cyclic permutation ordering. 
For example, let $ G $ be the butterfly graph pictured in Figure \ref{fig: butterfly graph} and define the edge ordering $ \omega $ by $ \omega \coloneqq (e_{1},e_{2},e_{3},e_{4},e_{5},e_{6}) $. 
Then 
\begin{align*}
\pi_{\omega} = \tau_{e_{6}} \tau_{e_{5}} \tau_{e_{4}} \tau_{e_{3}} \tau_{e_{2}}\tau_{e_{1}} 
= (3 \ 5) (1 \ 3) (4 \ 5) (3 \ 4) (2 \ 3) (1 \ 2)
= (1 \ 3 \ 2 \ 5 \ 4). 
\end{align*}
Therefore $ \omega $ is a full cyclic permutation ordering of the butterfly graph $ G $. 

Let $ \beta(G) \coloneqq |E_{G}| - |V_{G}| + 1 $ denote the \textbf{Betti number} (also called the \textbf{circuit rank}) of the connected graph $ G $. 
When a graph $ G $ has a full cyclic permutation ordering, considering the signature of a full cyclic permutation ordering, one can show that the Betti number $ \beta(G) $ is even. 
However, the converse is false. 
For example, the dumbbell graph (Figure \ref{fig: dumbbell graph}) has no full cyclic permutation orderings although its Betti number is $ 2 $.


We regard a graph as a topological space by identifying each edge with the unit interval $ [0,1] $ and gluing them at vertices. 
Then $\beta(G)$ coincides with the Betti number of $G$ as a topological space. 
In this article, we will show that having a full cyclic permutation ordering is a topological property as follows. 

Let $ \Sigma $ be an orientable closed surface and $ \iota \colon G \to \Sigma $ an embedding. 
We call a connected component of the complement of the image of $ \iota $ a \textbf{face}. 
An embedding $ \iota $ is a \textbf{$ 2 $-cell embedding} if every face is homeomorphic to an open disk. 
For any $ 2 $-cell embedding $ \iota \colon G \to \Sigma $, 
\begin{align*}
|V_{G}| - |E_{G}| + f_{\iota} = 2-2 g_{\Sigma}, 
\end{align*}
where $ f_{\iota} $ denotes the number of faces of the $ 2 $-cell embedding $ \iota $ and $ g_{\Sigma} $ the genus of $ \Sigma $. 
Then
\begin{align*}
2g_{\Sigma} + f_{\iota} = \beta(G) + 1. 
\end{align*}
Therefore maximizing the genus $ g_{\Sigma} $ is equivalent to minimizing the number of faces $ f_{\iota} $.
Hence the \textbf{maximum genus} $ \gamma_{\mathrm{max}}(G) $, the  maximum of genus $ g_{\Sigma} $ such that there exists a $ 2 $-cell embedding $ G \to \Sigma $, satisfies 
\begin{align*}
\gamma_{\mathrm{max}}(G) \leq \left\lfloor \dfrac{\beta(G)}{2} \right\rfloor, 
\end{align*}
where $ \lfloor \ \rfloor $ denotes the floor function. 
If the equality holds, then we say that $ G $ is \textbf{upper embeddable}. 

Upper embeddable graphs are well-studied objects by many researchers \cite{
huang2000face-joctsb
,jungerman1978characterization-totams
,nebesky1981every-jogt
,nebesky1983note-cmj
,nebesky1985locally-cmj
,payan1979upper-dm,
skoviera1991maximum-dm
,skoviera1992decay-ms
,skoviera1989maximum-dm
,xuong1979how-joctsb
,xuong1979upper-embeddable-joctsb
}. 
Jungerman and Xuong gave a combinatorial characterization of upper embeddability independently. 
\begin{theorem}[{Jungerman \cite[Theorem 2]{jungerman1978characterization-totams}}, {Xuong \cite[Theorem A]{xuong1979upper-embeddable-joctsb}}]\label{Jungerman, Xuong}
A connected graph $G$ with even (odd) Betti number is upper embeddable if and only if there exists a spanning tree $T$ of $G$ such that all (all but one) connected components of $G\setminus T$ consists of an even number of edges. 
\end{theorem}

Here is the main theorem of this article. 

\begin{theorem}\label{main theorem: full cyclic}
Given a graph $ G $, the following are equivalent. 
\begin{enumerate}[(1)]
\item\label{main theorem: full cyclic 1} $ G $ has a full cyclic permutation ordering. 
\item\label{main theorem: full cyclic 2} There exists a $ 2 $-cell embedding $ \iota \colon G \to \Sigma $ such that $ f_{\iota} = 1 $. 
\item \label{main theorem: full cyclic 3} The Betti number $ \beta(G) $ is even and $ G $ is upper embeddable, that is, $ \beta(G) = 2 \gamma_{\mathrm{max}}(G) $. 
\item \label{main theorem: full cyclic 4} There exists a spanning tree $ T $ of $ G $ such that every connected components of $ G \setminus T $ consists of an even number of edges. 
\end{enumerate}
\end{theorem}

Note that the conditions (\ref{main theorem: full cyclic 2}), (\ref{main theorem: full cyclic 3}), and (\ref{main theorem: full cyclic 4}) are equivalent by the definition of upper embeddability and Theorem \ref{Jungerman, Xuong}.
Figure \ref{fig: embedding} shows how the butterfly graph can be embedded into a torus with exactly one face. 

\begin{figure}[t]
\centering
\begin{tikzpicture}[baseline=-43]
\draw (-1.6,  0.5) node[v](1){} node[above]{$ 1 $};
\draw (-1.6, -0.5) node[v](2){} node[below]{$ 2 $};
\draw (     0,    0) node[v](3){} node[above]{$ 3 $};
\draw ( 1.6,  0.5) node[v](4){} node[above]{$ 4 $};
\draw ( 1.6, -0.5) node[v](5){} node[below]{$ 5 $};
\draw (1)--(2);
\draw (2)--(3);
\draw (3)--(4);
\draw (4)--(5);
\draw (1)--(3);
\draw (3)--(5);
\draw[right hook->] (3,0) -- (5,0);
\end{tikzpicture}
\qquad
\begin{tikzpicture}
\draw (0,0) ellipse (2.5 and 1.5);
\draw (-1.2,0.1) to[out=320,in=220] (1.2,0.1);
\draw (-1.08,0) to[out=40,in=140] (1.08,0);
\draw[very thick] (0,-0.35) to[out=330,in=30] (0,-1.5);
\draw[dotted,very thick] (0,-0.35) to[out=210,in=150] (0,-1.5);
\draw[very thick] (0,0) ellipse (1.8 and 1);
\draw (0.29,-0.98) node[circle, draw, inner sep=1.5pt, minimum size=2pt, fill=black](){};
\draw (0.22,-0.62) node[circle, draw, inner sep=1.5pt, minimum size=2pt, fill=black](){};
\draw (0.2,-1.3) node[circle, draw, inner sep=1.5pt, minimum size=2pt, fill=black](){};
\draw (0,-1) node[circle, draw, inner sep=1.5pt, minimum size=2pt, fill=black](){};
\draw (0.55,-0.95) node[circle, draw, inner sep=1.5pt, minimum size=2pt, fill=black](){};
\draw (0.38, -0.66) node[scale=0.5](){$ 1 $}; 
\draw (0.35, -1.4) node[scale=0.5](){$ 2 $}; 
\draw (0.38, -1.17) node[scale=0.5](){$ 3 $};
\draw (0, -1.2) node[scale=0.5](){$ 4 $}; 
\draw (0.6, -1.15) node[scale=0.5](){$ 5 $}; 
\end{tikzpicture}
\caption{A $ 2 $-cell embedding of the butterfly graph into a torus with exactly one face. }
\label{fig: embedding}
\end{figure}
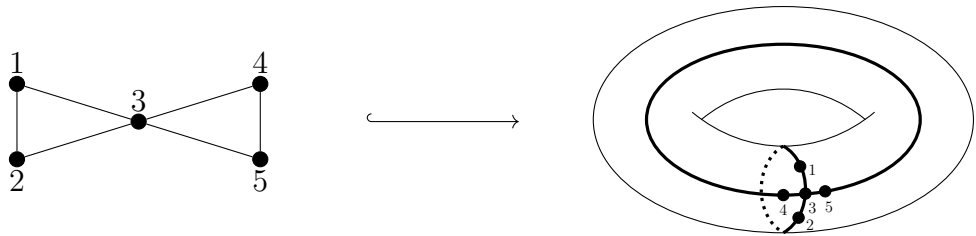

The organization of this article as follows. 
In Section \ref{section: Proof (3) => (1)}, we will prove that (\ref{main theorem: full cyclic 3}) implies (\ref{main theorem: full cyclic 1}) and give an example of constructing a full cyclic permutation ordering. 
In Section \ref{section: embedding and rotation system}, we will review the relation of $ 2 $-cell embeddings and rotation systems and we will prove that (\ref{main theorem: full cyclic 1}) implies (\ref{main theorem: full cyclic 2}). 
Combining the proofs in Section \ref{section: Proof (3) => (1)} and Section \ref{section: embedding and rotation system}, we will complete the proof of Theorem \ref{main theorem: full cyclic}. 

In Section \ref{section: identity permutation ordering}, we will study another extreme condition, that is, edge orderings $ \omega $ such that $ \pi_{\omega} $ is the identity permutation, which Eden \cite{eden1967relation-joct} studied. 
Eden gave necessary conditions for such orderings and asked whether the condition is also sufficient. 
We will give a counterexample for this question.

\section{Proof that (\ref{main theorem: full cyclic 3}) implies (\ref{main theorem: full cyclic 1})}\label{section: Proof (3) => (1)}

First, we introduce the following lemma. 
\begin{lemma}\label{uvw}
Let $\pi$ be a full cyclic permutation in $S_{n}$. 
If the distinct numbers $u,v,w$ appear in $\pi$ in this cyclic order, then the product $(u \ v)(v \ w) \pi$ is a full cyclic permutation. 
\end{lemma}
\begin{proof}
By the assumption, we can write 
\begin{align*}
\pi = (u \ a_{1} \ \cdots \ a_{r} \ v \ b_{1} \ \cdots \ b_{s} \ w \ c_{1} \ \cdots \ c_{t}), 
\end{align*}
where $a_{i}, b_{i}, c_{i}$ denote distinct numbers in $[n]\setminus \{u,v,w\}$.
Then we obtain
\begin{align*}
(u \ v)(v \ w)\pi 
&= (u \ v)(v \ w)(u \ a_{1} \ \cdots \ a_{r} \ v \ b_{1} \ \cdots \ b_{s} \ w \ c_{1} \ \cdots \ c_{t}) \\
&= (u \ a_{1} \ \cdots \ a_{r} \ w \ c_{1} \ \cdots \ c_{t} \ v \ b_{1} \ \cdots \ b_{s}),
\end{align*}
which is a full cyclic permutation.
\end{proof}

We say that two edges $e$ and $e^{\prime}$ are \textbf{adjacent} if $r_{G}(e) \cap r_{G}(e^{\prime}) \neq \varnothing$, that is, they have a common endvertex. 
Note that the case $r_{G}(e) = r_{G}(e^{\prime})$ is allowed. 

\begin{lemma}\label{Gp FCPO=>G FCPO}
Let $e$ and $e^{\prime}$ be two adjacent edges in a graph $G$. 
If $G \setminus \{e, e^{\prime}\}$ has a full cyclic permutation ordering, then $G$ has a full cyclic permutation ordering. 
\end{lemma}
\begin{proof}
Let $\omega^{\prime} = (e_{1}, \dots, e_{m})$ be a full cyclic permutation ordering of $G \setminus \{e, e^{\prime}\}$. 
If $r_{G}(e) = r_{G}(e^{\prime})$, then $\omega \coloneqq (e_{1}, \dots, e_{m}, e^{\prime}, e)$ is a full cyclic permutation ordering of $G$ since $\pi_{\omega} = \pi_{\omega^{\prime}}$. 

Now, suppose that $r_G(e) = \{u, v\}$ and $r_G(e^{\prime}) = \{v, w\}$ with $u \neq w$.
If the cycle order of $u,v,w$ in $\pi_{\omega^{\prime}}$ is $u,v,w$, then $\omega \coloneqq (e_{1}, \dots, e_{m}, e^{\prime}, e)$ is a full cyclic permutation ordering of $G$ by Lemma \ref{uvw}. 
In a symmetrical manner, if the cycle order is $w,v,u$, then let $\omega \coloneqq (e_{1}, \dots, e_{m}, e, e^{\prime})$. 
\end{proof}

The following lemma is required. 

\begin{lemma}[{Xoung \cite[Lemma 3]{xuong1979how-joctsb}}]\label{Xoung lemma}
Suppose that $G$ is upper embeddable and $\beta(G)$ is even. 
If $G$ is not a tree, then there exist two adjacent edges $e$ and $e^{\prime}$ such that $G\setminus \{e, e^{\prime}\}$ is upper embeddable. 
\end{lemma}

\begin{proof}[Proof that (\ref{main theorem: full cyclic 3}) implies (\ref{main theorem: full cyclic 1}) in Theorem \ref{main theorem: full cyclic}]
We will show that $ G $ has a full cyclic permutation ordering by induction on the Betti number $\beta(G)$.
When $ \beta(G) = 0 $, $ G $ is a tree and has a full cyclic permutation ordering by D\'enes' theorem (Theorem \ref{Denes}). 

Assume $ \beta(G) > 0$.
By Lemma \ref{Xoung lemma}, there exist two adjacent edges $e$ and $e^{\prime}$ such that $G^{\prime} \coloneqq G\setminus \{e, e^{\prime}\}$ is upper embeddable. 
By the induction hypothesis, $ G' $ has a full cyclic permutation ordering.
By Lemma \ref{Gp FCPO=>G FCPO}, $G$ has a full cyclic permutation ordering.
\end{proof}

\begin{example}
Consider the wheel graph $W_{5}$ pictured in Figure \ref{Fig: wheel graph}. 
The Betti number of $W_{5}$ is $4$. 
Let $T$ be the spanning tree of $W_{5}$ consisting of the edges $12, 23, 34, 45$. 
Then $W_{5}\setminus T$ is connected and consisting of $4$ edges. 
Therefore $W_{5}$ satisfies the condition (\ref{main theorem: full cyclic 4}) in Theorem \ref{main theorem: full cyclic} and hence has a full cyclic permutation ordering. 
We will construct a full cyclic permutation ordering following by the proof of Lemma \ref{Gp FCPO=>G FCPO}. 
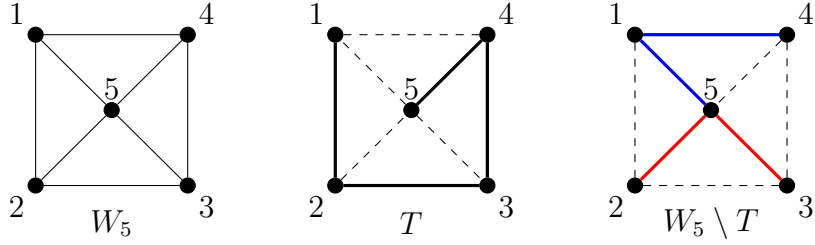
\begin{figure}[t]
\centering
\begin{tikzpicture}[baseline=0]
\draw (0,2) node[v](1){} node[above left](){$ 1 $};
\draw (0,0) node[v](2){} node[below left](){$ 2 $};
\draw (2,0) node[v](3){} node[below right](){$ 3 $};
\draw (2,2) node[v](4){} node[above right](){$ 4 $};
\draw (1,1) node[v](5){} node[above](){$ 5 $};
\draw (1) -- (2) -- (3) -- (4) -- (1) -- cycle; 
\draw (1)--(5)--(3);
\draw (2)--(5)--(4);
\draw (1,-0.5) node{$W_{5}$};
\end{tikzpicture}
\qquad
\begin{tikzpicture}[baseline=0]
\draw (0,2) node[v](1){} node[above left](){$ 1 $};
\draw (0,0) node[v](2){} node[below left](){$ 2 $};
\draw (2,0) node[v](3){} node[below right](){$ 3 $};
\draw (2,2) node[v](4){} node[above right](){$ 4 $};
\draw (1,1) node[v](5){} node[above](){$ 5 $};
\draw[very thick] (1) -- (2) -- (3) -- (4) -- (5); 
\draw[dashed] (2)--(5)--(3);
\draw[dashed] (5)--(1)--(4) ;
\draw (1,-0.5) node{$T$};
\end{tikzpicture}
\qquad
\begin{tikzpicture}[baseline=0]
\draw (0,2) node[v](1){} node[above left](){$ 1 $};
\draw (0,0) node[v](2){} node[below left](){$ 2 $};
\draw (2,0) node[v](3){} node[below right](){$ 3 $};
\draw (2,2) node[v](4){} node[above right](){$ 4 $};
\draw (1,1) node[v](5){} node[above](){$ 5 $};
\draw[dashed] (1) -- (2) -- (3) -- (4) -- (5); 
\draw[very thick, red] (2)--(5)--(3);
\draw[very thick, blue] (5)--(1)--(4) ;
\draw (1,-0.5) node{$W_{5}\setminus T$};
\end{tikzpicture}
\caption{Wheel graph $W_{5}$ and its spanning tree $T$}\label{Fig: wheel graph}
\end{figure}

We partition the edges of $W_{5}\setminus T$ into two adjacent pairs $\{25,35\}$ and $\{14,15\}$. 
Define the edge ordering $\omega_{1}$ of $T$ by $\omega_{1} \coloneqq (12,23,34,45)$. 
By D\'enes' theorem (Theorem \ref{Denes}), $\pi_{\omega_{1}}$ is a full cyclic permutation. 
Indeed we have 
\begin{align*}
\pi_{\omega_{1}} = (4 \ 5)(3 \ 4)(2 \ 3)(1 \ 2) = (2 \ 1 \ 5 \ 4 \ 3). 
\end{align*}
Next we define the edge ordering $\omega_{2}$ of $T \cup \{25,35\}$. 
Following the proof above, define $\omega_{2}$ by $\omega_{2} \coloneqq \omega_{1} \ast (35, 25)$, where $\ast$ denotes the concatenation. 
Then 
\begin{align*}
\pi_{\omega_{2}} = (2 \ 5)(3 \ 5) (2 \ 1 \ 5 \ 4 \ 3) = (4 \ 2 \ 1 \ 3 \ 5). 
\end{align*}
Similarly, define $\omega_{3}$ by $\omega_{3} \coloneqq \omega_{2} \ast (15, 14)$. 
Then 
\begin{align*}
\pi_{\omega_{3}} = (1 \ 4)(1 \ 5)(4 \ 2 \ 1 \ 3 \ 5) = (4 \ 2 \ 5 \ 1 \ 3). 
\end{align*}
Thus we obtain a full cyclic permutation ordering $\omega_{3} = (12, 23, 34, 45, 35, 25, 15, 14)$ of $W_{5}$. 
\end{example}

\section{Proof that (\ref{main theorem: full cyclic 1}) implies (\ref{main theorem: full cyclic 2})}\label{section: embedding and rotation system}

Let $ I_{G}(v) $ denote the set of edges of $ G $ incident to a vertex $ v \in V_{G} $. 
A \textbf{rotation system} of $ G $ is a collection $ \rho = (\rho_{v})_{v \in V_{G}} $ consisting of cyclic orders $ \rho_{v} $ on $ I_{G}(v) $, where a cyclic order on $ I_{G}(v) $ is an equivalence class of linear orders on $ I_{G}(v) $ obtained by identifying $ (e_{1}, e_{2}, \dots, e_{s}) $ with its circular shift $ (e_{2}, \dots, e_{s}, e_{1}) $, denoted by $ [e_{1}, \dots, e_{s}] $. 

Every embedding of $ G $ on an orientable closed surface defines a rotation system with the clockwise ordering for each vertex. 
Conversely, from a rotation system, we can obtain a $ 2 $-cell embedding of $ G $ on an orientable closed surface as follows. 

Define $ D_{G} $ by 
\begin{align*}
D_{G} \coloneqq \Set{(e,u) \in E_{G} \times V_{G} | u \in r_{G}(e)}. 
\end{align*}
We call an element of $ D_{G} $ a \textbf{dart}. 
When $r_{G}(e) = \{u,v\}$, the dart $ (e,u) $ shows an orientation of the edge $ e $ from the \textbf{source} $ u $ to the \textbf{target} $ v $. 
Define the involution $ \alpha $ on $ D_{G} $ by $ \alpha(e,u) \coloneqq (e,v) $. 

Given a rotation system $ \rho = (\rho_{v})_{v \in V_{G}} $, we will define bijections $ \sigma $ and $ \phi $ from $ D_{G} $ to itself. 
Suppose that $ \rho_{v} = [e_{1}, \dots, e_{s}] $. 
Define $ \sigma $ by $ \sigma(e_{i}, v) \coloneqq (e_{i+1}, v) $, where we consider $ e_{s+1} = e_{1} $. 
Let $ \phi \coloneqq \sigma \circ \alpha $. 

For every dart $ d $, the target of $ d $ coincides with the source of $ \phi(d) $.
Therefore each orbit in $ D_{G}/\langle \phi \rangle $ determines a closed walk on $ G $ and we can make a polygon whose sides are formed by the darts in the orbit. 
Gluing the sides of the polygons obtained from the orbits in $ D_{G}/\langle \phi \rangle $ by the involution $ \alpha $, we obtain an embedding of $ G $ on a closed surface. 
One can show that this surface is actually orientable (See \cite[Subsection 3.2]{mohar2001graphs}) and hence this embedding is the desired $ 2 $-cell embedding. 

\begin{theorem}[See {\cite[Theorem 3.2.4]{mohar2001graphs}}]
Given a graph $ G $, there exists a one-to-one correspondence between rotation systems of $ G $ and $ 2 $-cell embeddings of $ G $ on oriented closed surfaces up to orientation-preserving homeomorphism. 
\end{theorem}

Note that, from the construction, the number of the faces of the embedding corresponding to a rotation system is equal to the number of the orbits in $ D_{G}/\langle \phi \rangle $. 

Let $ \omega $ be an edge ordering of $ G $. 
For each $ v \in V_{G} $, let $ \omega_{v} $ denote the linear order on $ I_{G}(v) $ induced by $ \omega $. 
Moreover, let $ \rho_{\omega, v} $ be the cyclic order on $ I_{G}(v) $ determined by $ \omega_{v} $. 
Thus we obtain the rotation system $ \rho_{\omega} \coloneqq (\rho_{\omega, v})_{v \in V_{G}} $ from an edge ordering $ \omega $ and hence the corresponding bijections $ \sigma_{\omega} $ and $ \phi_{\omega} = \sigma_{\omega} \circ \alpha $.

\begin{lemma}\label{edge ordering to rotation system}
Let $ \omega = (e_{1}, \dots, e_{m}) $ be an edge ordering of a graph $ G $. 
For each $ v \in V_{G} $, let $ f_{v} $ denote the minimal edge in $I_{G}(v)$ with respect to $ \omega $. 
Define a map $ \Psi \colon V_{G}/\langle \pi_{\omega} \rangle \to D_{G}/\langle \phi_{\omega} \rangle $ by 
\begin{align*}
\Psi([v]) \coloneqq [(f_{v}, v)], 
\end{align*}
where the brackets denote equivalence classes. 
Then $ \Psi $ is a bijection. 
\end{lemma}
\begin{proof}
First, we will show that the map $ \Psi $ is well-defined. 
It is sufficient to show that $[(f_{v},v)] = [(f_{\pi_{\omega}(v)}, \pi_{\omega}(v))]$ for each $v \in V_{G}$.

Fix $ v \in V_{G} $. 
The edge ordering $\omega = (e_{1}, \dots, e_{m})$ defines the set $T_{v}$ as follows. 
\begin{align*}
T_{v} \coloneqq \Set{e_{i} \in E_{G} | (\tau_{i}\tau_{i-1} \cdots \tau_{1})(v) \neq (\tau_{i-1} \cdots \tau_{1})(v) }, 
\end{align*}
where we agree with $(\tau_{i-1} \cdots \tau_{1})(v) = v $ if $i = 1$. 
Suppose that $T_{v}  = \Set{e_{j_{1}}, e_{j_{2}}, \dots, e_{j_{s}}}$ with $ j_{1} < j_{2} < \dots <j_{s} $. 
Then $(e_{j_{1}}, \dots, e_{j_{s}})$ is a trail from $v$ to $\pi_{\omega}(v)$. 
Let $ v_{0} \coloneqq v $ and for $k \in \{1, \dots, s\}$ define $v_{k}$ recursively as the endvertex of $e_{j_{k}}$ other than $v_{k-1}$. 
Note that $v_{s} = \pi_{\omega}(v)$. 

From the definition of $T_{v}$, for each $k \in \{1, \dots, s-1\}$,  $e_{j_{k+1}}$ covers $e_{j_{k}}$ in $I_{G}(v_{k})$ with respect to $\omega$. 
Therefore $\phi_{\omega}(e_{j_{k}}, v_{k-1}) = (e_{j_{k+1}},v_{k})$. 
Since $e_{j_{s}}$ is the maximal element in $I_{G}(v_{s})$ with respect to the order $\omega$, we have $\phi_{\omega}(e_{j_{s}}, v_{s-1}) = (f_{v_{s}}, v_{s}) = (f_{\pi_{\omega}(v)}, \pi_{\omega}(v))$. 
Thus $\Psi$ is well-defined. 

Next, to prove the surjectivity, take an orbit $ W \in D_{G}/\langle \phi_{\omega} \rangle $. 
Let $f$ be the minimal element in 
\begin{align*}
\Set{e \in E_{G} | (e,v) \in W \text{ for some } v \in V_{G}}
\end{align*}
with respect to $\omega$. 
Suppose $(f,v) \in W$. 
Assume $ f $ is not minimal in $ I_{G}(v) $ with respect to $ \omega_{v} $. 
Then $ f^{\prime} \coloneqq \sigma_{\omega}^{-1}(f) $ is less than $ f $  in $I_{G}(v)$ with respect to $ \omega $.
Let $r_{G}(f^{\prime}) = \{v,v^{\prime}\}$.
Then $\phi_{\omega}(f^{\prime}, v^{\prime}) = (f,v)$ and hence $(f^{\prime}, v^{\prime}) \in W$.  
This contradicts to the minimality of $ f $. 
Therefore $ f $ is minimal in $ I_{G}(v) $ and hence $f = f_{v}$. 
Hence $ \Psi([v]) = [(f, v)] = W $. 

Finally, we prove the injectivity. 
Let $ u,v \in V_{G} $ and suppose that $ \Psi([u]) = \Psi([v]) $. 
Then we have $ \phi_{\omega}^{s}(f_{u}, u) = (f_{v}, v) $ for some $ s \in \mathbb{Z} $. 
Without loss of generality, we can assume that $ s > 0 $. 

Recall the edges of $G$ are ordered by $\omega = (e_{1}, \dots, e_{m})$. 
We can write the edge $f_{u}$ as $f_{u} = e_{j_{0}}$ with some $j_{0} \in \{1, \dots, m\}$. 
Moreover we can obtain the walk $(e_{j_{0}}, e_{j_{1}}, \dots, e_{j_{s}})$ by $(e_{j_{k}}, v_{k}) \coloneqq \phi_{\omega}(e_{j_{k-1}}, v_{k-1})$ for $k \in \{1, \dots, s\}$, where $v_{0} \coloneqq u$. 
Note that $ (e_{j_{0}}, v_{0}) = (f_{u}, u) $ and $ (e_{j_{s}}, v_{s}) = (f_{v}, v) $. 
Suppose that 
\begin{align*}
\Set{k \in \{1, \dots, s\} | j_{k-1} \geq j_{k} }  = \{p_{1}, \dots, p_{t}\}
\end{align*}
with $ p_{1} < \dots < p_{t} = s $. 
Then $ \pi_{\omega}^{i}(v_{0}) = v_{j_{p_{i}}} $ for $ i \in \{1, \dots, t\} $. 
In particular, $ \pi_{\omega}^{t}(u) = \pi_{\omega}^{t}(v_{0}) = v_{j_{p_{t}}} = v_{s} = v $. 
Thus $ [u] = [v] $ and hence $ f $ is injective. 
\end{proof}

Now we are ready to prove that (\ref{main theorem: full cyclic 1}) implies (\ref{main theorem: full cyclic 2}). 
\begin{proof}[Proof that (\ref{main theorem: full cyclic 1}) implies (\ref{main theorem: full cyclic 2}) in Theorem \ref{main theorem: full cyclic}]
Let $\omega$ be a full cyclic permutation ordering of $G$. 
Then the number of faces of the $2$-cell embedding corresponding to the rotation system $\rho_{\omega}$ is equal to $|D_{G}/\langle \phi_{\omega}\rangle| = |V_{G}/\langle\pi_{\omega}\rangle| = 1$. 
\end{proof}

\section{Identity permutation ordering}\label{section: identity permutation ordering}

In this section, a graph is not necessarily connected. 
We say that an edge ordering $ \omega $ of a graph $ G $ is an \textbf{identity permutation ordering} if $ \pi_{\omega} = \varepsilon $, where $ \varepsilon $ denotes the identity permutation. 
Every edgeless graph vacuously has an identity permutation ordering. 
The minimal example of a non-trivial graph having an identity permutation ordering is the $ 2 $-cycle $ C_{2} $. 
Eden \cite{eden1967relation-joct} studied simple graphs that have an identity permutation ordering and mentioned the complete graph $ K_{4} $ is the minimal example. 
Figure \ref{fig: C2 and K4} shows identity permutation orderings of $ C_{2} $ and $ K_{4} $. 

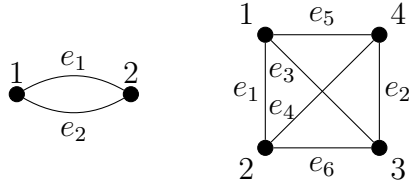
\begin{figure}
\centering
\begin{tikzpicture}[baseline=-36]
\draw (0,0) node[v](1){} node[above](){$ 1 $};
\draw (1.5,0) node[v](2){} node[above](){$ 2 $};
\draw (0.75, 0.45) node(){$ e_{1} $};
\draw (0.75,-0.5) node(){$ e_{2} $};
\draw (1) to[out=30, in=150] (2);
\draw (1) to[out=330, in=210] (2);
\end{tikzpicture}
\qquad
\begin{tikzpicture}
\draw (0,1.5) node[v](1){} node[above left](){$ 1 $};
\draw (0,0) node[v](2){} node[below left](){$ 2 $};
\draw (1.5,0) node[v](3){} node[below right](){$ 3 $};
\draw (1.5,1.5) node[v](4){} node[above right](){$ 4 $};
\draw (1) -- (2) node[midway, xshift=-7]{$ e_{1} $}; 
\draw (1) -- (3) node[midway, xshift=-15, yshift=7]{$ e_{3} $}; 
\draw (1) -- (4) node[midway, yshift=7]{$ e_{5} $}; 
\draw (2) -- (3) node[midway, yshift=-7]{$ e_{6} $}; 
\draw (2) -- (4) node[midway, xshift=-15, yshift=-7]{$ e_{4} $}; 
\draw (3) -- (4) node[midway, xshift=7]{$ e_{2} $}; 
\end{tikzpicture}
\caption{The $ 2 $-cycle $ C_{2} $ and the complete graph $ K_{4} $ have an identity permutation ordering.}
\label{fig: C2 and K4}
\end{figure}

Eden gave necessary conditions (without proof) for simple connected graphs having an identity permutation ordering as follows. 
\begin{proposition}[{Eden \cite[P. 130]{eden1967relation-joct}}]\label{Eden's conditions}
Let $ G $ be a simple connected graph on $ n $ vertices with $ m $ edges. 
If $ G $ has an identity permutation ordering, then the following conditions hold. 
\begin{enumerate}[(1)]
\item $ m $ is even. 
\item There exist a set $ \mathcal{C} $ consisting of closed trails and a map $ \psi \colon V_{G} \to \mathcal{C}$ such that the following conditions hold. 
\begin{enumerate}[(i)]
\item $ \psi $ is bijective. 
\item Every $ v \in V_{G} $ belongs to the closed trail $ \psi(v) $. 
\item The sum of the number of edges of closed trails in $ \mathcal{C} $ is $ 2m $. 
\item Each edge of $ G $ belongs to exactly two closed trails in $ \mathcal{C} $. 
\end{enumerate}
\end{enumerate}
\end{proposition}
\begin{proof}
Suppose that $ \omega = (e_{1}, \dots, e_{m}) $ be an identity permutation ordering. 
Since the identity permutation is an even permutation, we have $ m $ is even. 
Let $ \Psi \colon V_{G} \to D_{G}/\langle \phi_{\omega} \rangle $ be the bijection considered in Lemma \ref{edge ordering to rotation system}. 
Note that $ V_{G} = V_{G}/\langle \pi_{\omega} \rangle $ since $ \pi_{\omega} = \varepsilon $. 

For each $ v \in V_{G} $, there exists no dart $ d \in D_{G} $ such that both $ d $ and $ \alpha(d) $ belong to $ \Psi(v) $ since $ \pi_{\omega} $ is the identity. 
Thus, forgetting the direction of each dart in $ \Psi(v) $, we obtain the closed trail $ \psi(v) \subseteq E_{G} $. 
Letting $ \mathcal{W} \coloneqq \Set{\psi(v) | v \in V_{G}} $, we have a surjection $ \psi \colon V_{G} \to \mathcal{W} $. 

We will prove $ \psi $ is injective. 
Assume that there exist distinct vertices $ u,v $ such that $ \psi(u) = \psi(v) $. 
Let $ \omega^{\prime} = (f_{1}, \dots, f_{r}) $ be the induced order of $ \omega $ on $ \psi(u) $. 
Since $ \pi_{\omega^{\prime}}(u) = u $, the edges $ f_{1} $ and $ f_{r} $ are incident to $ u $. 
Also, $ f_{1} $ and $ f_{r} $ are incident to $ v $ by the same reason. 
Hence $ f_{1} $ and $ f_{r} $ are parallel edges between $ u $ and $ v $. 
This contradicts that $ G $ is simple. 
Therefore $ \psi $ is injective and hence bijective. 

By the definition of maps $ \Psi $ and $ \psi $, every $ v \in V_{G} $ belongs to $ \psi(v) $. 
Moreover, 
\begin{align*}
\sum_{v \in V_{G}}|\psi(v)| = \sum_{v \in V_{G}}|\Psi(v)| = |D_{G}| = 2m. 
\end{align*}
For any $ e \in E_{G} $, the two darts on $ e $ belongs distinct orbits $ \Psi(u) $ and $ \Psi(v) $. 
Then $ e $ belongs to $ \psi(u) $ and $ \psi(v) $ and the other trails do not contain $ e $. 
Thus the map $ \psi \colon V_{G} \to \mathcal{C} $ has the desired properties. 
\end{proof}

\begin{example}
For the complete graph $ K_{4} $ in Figure \ref{fig: C2 and K4}, the following map $ \psi $ satisfies the conditions in Proposition \ref{Eden's conditions}. 
\begin{align*}
\psi(1) = \{e_{1}, e_{4}, e_{5}\}, \quad
\psi(2) = \{e_{1}, e_{3}, e_{6}\}, \quad
\psi(3) = \{e_{2}, e_{4}, e_{6}\}, \quad
\psi(4) = \{e_{2}, e_{3}, e_{5}\}. 
\end{align*}
\end{example}

Eden asked whether the necessary conditions in Proposition \ref{Eden's conditions} are also sufficient. 
We will give a counterexample for this question. 
Let $ G $ be the graph on $ 12 $ vertices with $ 20 $ edges pictured in Figure \ref{fig: counterexample for Eden's question}.
\begin{figure}[t]
\centering
\begin{tikzpicture}
\draw (-1,2) node[v](1){} node[above]{$ v_{1} $};
\draw (-2,1) node[v](2){} node[left]{$ v_{2} $};
\draw ( 1,2) node[v](3){} node[above]{$ v_{3} $};
\draw ( 2,1) node[v](4){} node[right]{$ v_{4} $};
\draw (-1,-2) node[v](5){} node[below]{$ v_{5} $};
\draw (-2,-1) node[v](6){} node[left]{$ v_{6} $};
\draw ( 1,-2) node[v](7){} node[below]{$ v_{7} $};
\draw ( 2,-1) node[v](8){} node[right]{$ v_{8} $};
\draw (-1,0) node[v](9){} node[left]{$ v_{9} $};
\draw ( 0,1) node[v](10){} node[yshift = 10]{$ v_{10} $};
\draw ( 1,0) node[v](11){} node[right]{$ v_{11} $};
\draw ( 0,-1) node[v](12){} node[yshift = -13]{$ v_{12} $};
\draw (2) -- (9) -- (10) -- (2) -- (1) -- (10) -- (3) -- (4) -- (10) -- (11) -- (4); 
\draw (6) -- (5) -- (12) -- (6) -- (9) -- (12) -- (11) -- (8) -- (12) -- (7) -- (8); 
\end{tikzpicture}
\caption{A counterexample for Eden's question.}
\label{fig: counterexample for Eden's question}
\end{figure}
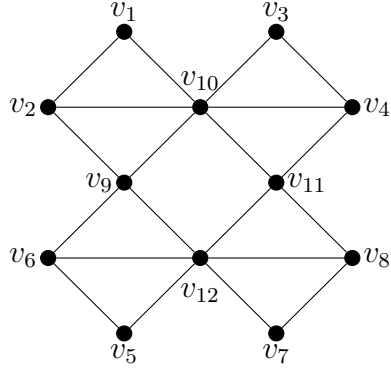 
Define a map $ \psi $ by 
\begin{align*}
\psi(v_{1}) &= v_{1}v_{2}v_{10}v_{1}, \quad
\psi(v_{2}) = v_{2}v_{9}v_{10}v_{2}, \quad
\psi(v_{3}) = v_{3}v_{10}v_{4}v_{3}, \quad
\psi(v_{4}) = v_{4}v_{10}v_{11}v_{4}, \\
\psi(v_{5}) &= v_{5}v_{6}v_{12}v_{5}, \quad
\psi(v_{6}) = v_{6}v_{9}v_{12}v_{6}, \quad
\psi(v_{7}) = v_{7}v_{8}v_{12}v_{7}, \quad
\psi(v_{8}) = v_{8}v_{11}v_{12}v_{8}, \\
\psi(v_{9}) &= v_{9}v_{2}v_{1}v_{10}v_{9}, \
\psi(v_{10}) = v_{10}v_{3}v_{4}v_{11}v_{10}, \
\psi(v_{11}) = v_{11}v_{8}v_{7}v_{12}v_{11}, \
\psi(v_{12}) = v_{12}v_{5}v_{6}v_{9}v_{12}. 
\end{align*}

The conditions in Proposition \ref{Eden's conditions} are satisfied. 
Suppose that $ G $ has an identity permutation ordering $ \omega $. 
By Lemma \ref{edge ordering to rotation system} there exists a $ 2 $-cell embedding $ \iota $ with $ f_{\iota} = |D_{G}/\langle \phi_{\omega} \rangle| = |V_{G}| = 12 $ faces. 
Then the genus $ g_{\iota} $ satisfies 
\begin{align*}
2-2g_{\iota} = |V_{G}| - |E_{G}| + f_{\iota} = 12 - 20 + 12 = 4. 
\end{align*}
Therefore $ g_{\iota} = -1 $, which is a contradiction. 

\begin{problem}
Characterize graphs that have an identity permutation ordering. 
\end{problem}

\section*{Acknowledgment}
The authors wish to thank Professor Sachiko Saito for her valuable comments about topics of topology. 
The authors would also like to express their deepest appreciation to the anonymous referee for careful reading and indicating deficiency in a proof. 

\bibliographystyle{amsplain}
\bibliography{bibfile}

\providecommand{\bysame}{\leavevmode\hbox to3em{\hrulefill}\thinspace}
\providecommand{\MR}{\relax\ifhmode\unskip\space\fi MR }
\providecommand{\MRhref}[2]{%
  \href{http://www.ams.org/mathscinet-getitem?mr=#1}{#2}
}
\providecommand{\href}[2]{#2}
\begin{thebibliography}{10}

\bibitem{denes1959representation-potmiothaos}
J.~D{\'e}nes, \emph{The representation of a permutation as the product of a
  minimal number of transpositions and its connection with the theory of
  graphs}, Publications of the Mathematical Institute of the Hungarian Academy
  of Sciences \textbf{4} (1959), 63--70.

\bibitem{eden1967relation-joct}
M.~Eden, \emph{On a relation between labeled graphs and permutations}, Journal
  of Combinatorial Theory \textbf{2} (1967), no.~2, 129--134.

\bibitem{foley2022transplanting}
A.~M. Foley, J.~Kazdan, L.~Kr{\"o}ll, S.~M. Alberga, O.~Melnyk, and
  A.~Tenenbaum, \emph{Transplanting {Trees}: {Chromatic} {Symmetric} {Function}
  {Results} through the {Group} {Algebra} of {${S}_n$}}, January 2022,
  arXiv:2112.09937 [math].

\bibitem{huang2000face-joctsb}
Y.~Huang and Y.~Liu, \emph{Face {Size} and the {Maximum} {Genus} of a {Graph}
  1. {Simple} {Graphs}}, Journal of Combinatorial Theory, Series B \textbf{80}
  (2000), no.~2, 356--370.

\bibitem{jungerman1978characterization-totams}
M.~Jungerman, \emph{A characterization of upper-embeddable graphs},
  Transactions of the American Mathematical Society \textbf{241} (1978),
  401--406.

\bibitem{mohar2001graphs}
B.~Mohar and C.~Thomassen, \emph{Graphs on surfaces}, Johns {Hopkins} studies
  in the mathematical sciences, Johns Hopkins University Press, Baltimore,
  2001.

\bibitem{moszkowski1989solution-ejoc}
P.~Moszkowski, \emph{A {Solution} to a {Problem} of {D{\'e}nes}: a {Bijection}
  {Between} {Trees} and {Factorizations} of {Cyclic} {Permutations}}, European
  Journal of Combinatorics \textbf{10} (1989), no.~1, 13--16.

\bibitem{nebesky1981every-jogt}
L.~Nebesk{\'y}, \emph{Every connected, locally connected graph is upper
  embeddable}, Journal of Graph Theory \textbf{5} (1981), no.~2, 205--207.

\bibitem{nebesky1983note-cmj}
\bysame, \emph{A note on upper embeddable graphs}, Czechoslovak Mathematical
  Journal \textbf{33} (1983), no.~1, 37--40.

\bibitem{nebesky1985locally-cmj}
\bysame, \emph{On locally quasiconnected graphs and their upper embeddability},
  Czechoslovak Mathematical Journal \textbf{35} (1985), no.~1, 162--166.

\bibitem{pawlowski2022chromatic-ac}
B.~Pawlowski, \emph{Chromatic symmetric functions via the group algebra of
  {${S}_n$}}, Algebraic Combinatorics \textbf{5} (2022), no.~1, 1--20.

\bibitem{payan1979upper-dm}
C.~Payan and N.~H. Xuong, \emph{Upper embeddability and connectivity of
  graphs}, Discrete Mathematics \textbf{27} (1979), no.~1, 71--80.

\bibitem{skoviera1991maximum-dm}
M.~{\v S}koviera, \emph{The maximum genus of graphs of diameter two}, Discrete
  Mathematics \textbf{87} (1991), no.~2, 175--180.

\bibitem{skoviera1992decay-ms}
\bysame, \emph{The decay number and the maximum genus of a graph}, Mathematica
  Slovaca \textbf{42} (1992), 391--406.

\bibitem{skoviera1989maximum-dm}
M.~{\v S}koviera and R.~Nedela, \emph{The maximum genus of vertex-transitive
  graphs}, Discrete Mathematics \textbf{78} (1989), no.~1-2, 179--186.

\bibitem{uchiumi2023signed-ejogtaa}
R.~Uchiumi, \emph{Signed graphs and signed cycles of hyperoctahedral groups},
  Electronic Journal of Graph Theory and Applications \textbf{11} (2023),
  no.~2, 419.

\bibitem{xuong1979how-joctsb}
N.~H. Xuong, \emph{How to determine the maximum genus of a graph}, Journal of
  Combinatorial Theory, Series B \textbf{26} (1979), no.~2, 217--225.

\bibitem{xuong1979upper-embeddable-joctsb}
\bysame, \emph{Upper-embeddable graphs and related topics}, Journal of
  Combinatorial Theory, Series B \textbf{26} (1979), no.~2, 226--232.

\end{thebibliography}

\end{document}